\documentclass[12pt,oneside,reqno]{amsart}
\usepackage[all]{xy}
\usepackage{amsfonts,amsmath,oldgerm,amssymb,amscd}
\UseComputerModernTips
\numberwithin{equation}{section}
\usepackage[breaklinks]{hyperref}
\allowdisplaybreaks

\usepackage{enumerate}

\usepackage{caption} 
\newtheorem{theorem}{Theorem}
\newtheorem{prop}{Proposition}

\newtheorem{lemma}[subsection]{{\bf Lemma}}
\newtheorem{coro}[subsection]{{\bf Corollary}}

\newtheorem{remark}[subsection]{Remark}

\begin{document}

\title[Analytic continuation of $\ell$-generalized Fibonacci zeta function]{Analytic continuation of $\ell$-generalized Fibonacci zeta function} 

\author[D. K. Sahoo]{D. K. Sahoo}
\address{Dilip Kumar Sahoo, Indian Institute of Science Education and Research, Berhampur, Odisha - 760010\\ India.}
\email{dilipks18@iiserbpr.ac.in}

\author[N.K.Meher]{N.K. Meher}
\address{Nabin Kumar Meher, Department of Mathematics, Indian Institute of Information Technology Raichur, Govt. Engineering College, Yermarus Camp, Raichur, 584135, Karnataka, India.}
\email{mehernabin@gmail.com, nabinmeher@iiitr.ac.in}


\pagenumbering{arabic}
\pagestyle{headings}

\begin{abstract}
In this paper, for any positive integer $\ell\geq2,$ we define $\ell$-generalized Fibonacci zeta function. We then study its analytic continuation to the whole complex plane $\mathbb{C}.$ Further, we compute a possible list of singularities and residues of the function at these simple poles. Moreover, we deduce that the special values of $\ell$-generalized Fibonacci zeta function at negative integer arguments are rational.
\end{abstract}
\maketitle
\section{Introduction}
Let $\ell \geq 2$ be an integer. The $n^{th}$ generalized Fibonacci sequence $\left( F_n^{(\ell)}\right)_{n \geq 2-\ell} $ defined as 
\begin{equation*} 
	F_n^{(\ell)} = F_{n-1}^{(\ell)} + F_{n-2}^{(\ell)} + \cdots + F_{n-\ell}^{(\ell)}
\end{equation*}
with the initial conditions 
 \begin{align*}
 F_{-(\ell -2)}^{(\ell)} =F_{-(\ell -3)}^{(\ell)} = \cdots = F_0^{(\ell)} = 0   \quad \hbox{and} \quad F_1^{(\ell)} = 1. 
 \end{align*}
 
$F_n^{(\ell)} $ is also called as $n^{th}$ $\ell$-generalized Fibonacci number. 
 From \cite{Dda2020}, we obtain that 
  \begin{align*}
  F_n^{(\ell)} = 2^{n-2} \quad \hbox{for all}\quad  2\leq n \leq \ell + 1, \quad \hbox{and } \quad  F_n^{(\ell)} < 2^{n-2} \quad \hbox{for all} \quad  n \geq \ell+2.
  \end{align*}
The characteristics polynomial of the $\ell$-generalized Fibonacci number is given by 
\begin{align}\label{eq3}
\phi_{\ell}(x) = x^{\ell}-x^{\ell -1} - \cdots - x -1.
\end{align}
It is irreducible over $\mathbb{Q}[x]$ and has one root outside the unit circle. Let $\alpha= \alpha_1$ be that single root which lies in between $2 \left( 1-2^{-\ell}\right)$ and $2$ (see \cite{Wolfram1998}) and it is the dominant root of $\phi_{\ell}(x). $ The other roots of the polynomial \eqref{eq3} are $ \alpha_2, \ldots, \alpha_{\ell}.$ When $\ell$ is an even integer, then $\phi_{\ell}(x)$ has one negative real root which lies in the interval $(-1,0).$ In 2014, G. P. B. Dresden and Z. Du \cite{Dre2014} gave the ``Binet like formula for the terms $F_n^{(\ell)} $ "  which is given by 
\begin{align}\label{eq5}
F_n^{(\ell)} = \sum_{i=1}^{\ell} \frac{\alpha_{i} -1}{2+ (\ell+1) (\alpha_{i} -2)} \alpha_{i}^{n-1}.
\end{align}
From \cite{Dda2020}, it is also known that 
\begin{align*}
\left| F_n^{(\ell)} - \frac{\alpha -1}{2+ (\ell+1) (\alpha -2)} \alpha^{n-1} \right|< \frac{1}{2} \quad \hbox{ for all} \quad n \geq 2-\ell.
\end{align*}
In 2013, J. J. Bravo and F. Luca \cite{Bravo2013} obtained that 
\begin{align}\label{eq6}
\alpha^{n-2} \leq  F_n^{(\ell)} \leq \alpha^{n-1} \quad \hbox{holds for all} \quad n \geq 1 \quad  \hbox{and} \quad  \ell \geq 2.
\end{align}

When $\ell=2$, $F_n^{(\ell)}$ is same as Fibonacci number $F_n$ and when $\ell=3,$ it coincides with Tribonacci numbers $T_n$.

The Fibonacci zeta function is defined by the series 
 \begin{equation*}
 \zeta_{F}(s) = \sum_{n=1}^{\infty} \frac{1}{F_n^{s}}, \quad \mathrm{Re}(s) >0.
 \end{equation*}
 Analytic continuation of Fibonacci zeta function was studied by L. Navas \cite{Navas2001} in 2001. In 2013, K. Kamano \cite{Kamano2013} studied the analytic continuation of Lucas zeta function. Further, analytic continuation of the multiple Fibonacci and Lucas zeta functions were studied by the second author and S. S. Rout in ( \cite{Meher2018}, \cite{Meher2021}, \cite{Rout2018}). The arithmetic nature of special values of Fibonacci zeta function and the special value of Riemann zeta function $\zeta(s)$ behaves similar in nature. The irrationality of $ \zeta_{F}(1)$ was proved by R. Andr$\acute{e}$-Jeannin \cite{Andre1989} in 1989, while the the transcendence of $ \zeta_{F}(2m), \  \hbox{for} \  m \in \mathbb{N}$ was given by D. Duverney et.al. \cite{Durveny1997} in 1997. Furthermore, in 2007, C. Elsner et.al. \cite{Elsner2007} showed that $\zeta_{F}(2), \zeta_{F}(4), \zeta_{F}(6)$ are algebraically independent over $\mathbb{Q}$. Likewise, M. Ram Murty \cite{Murty2013} deduced that $ \zeta_{F}(2m), \ \hbox{for} \  m \in \mathbb{N}$  is transcendental by using the theory of modular form and a result of Yu. V. Nesterenko \cite{Nesterenko1996}.
 
 In this paper, we introduce the $\ell$-generalized Fibonacci zeta function which is defined as
 \begin{equation*}
 \zeta_{F^{(\ell)}}(s) = \sum_{n=1}^{\infty}  \frac{1}{ \left(F_n^{(\ell)}\right)^{s}}.
 \end{equation*}  
  When $\ell = 2,$ it is same as Fibonacci zeta function $\zeta_F(s).$ In this paper,  we study the analytic continuation of the $\ell$-generalized Fibonacci zeta function with a possible list of poles and their corresponding residues. Moreover, we discuss the arithmetic nature of $\ell$-generalized Fibonacci zeta function at negative integer arguments. The paper is organized as follows. In Section $2,$ we prove that $  \zeta_{F^{(\ell)}}$ is absolute convergent in $\mathrm{Re}(s) >0.$ In Section $3,$ we obtain the analytic continuation of $\ell$-generalized Fibonacci zeta function $  \zeta_{F^{(\ell)}}$, compute a possible list of poles and residues at these poles. In Section $4,$ we deduce that the special values of  $\ell$-generalized Fibonacci zeta function at negative integer arguments are rational.
  \section{Preliminaries}
 \begin{prop}\label{prop1}
 	The infinite series $\sum_{n >0} \left(F_n^{(\ell)}\right)^{-s}$ converges absolutely in the right half plane $\mathrm{Re} (s) >0.$
 \end{prop}
\begin{proof}
 From \eqref{eq6}, we get 
	 \begin{align}\label{eq7}
	\left| \left(F_n^{(\ell)}\right)^{-s} \right|  = \left(F_n^{(\ell)}\right)^{-\sigma} \leq \left(\alpha^{n-2}\right)^{-\sigma} = \alpha^{2 \sigma} \left(\alpha^{- n \sigma}\right).
	 \end{align}
		Since $ \sigma = \mathrm{Re} s > 0,$ from \eqref{eq7}, we obtain
	\begin{align*}
	\sum_{n=1}^{\infty}  \left| \left(F_n^{(\ell)}\right)^{-s} \right| \leq  \alpha^{2 \sigma} \sum_{n=1}^{\infty} \left(\alpha^{- n \sigma}\right) = \frac{\alpha^{2 \sigma}}{\alpha^{ \sigma}-1} < \infty.
	\end{align*}
	\end{proof}
The next lemma describes that the integer closest to first term of Binet-like formula is the $\ell$-generalized Fibonacci number.
\begin{lemma}[G. P. B. Dresden and Z. Du \cite{Dre2014}]\label{lem1}
	Let $F_n^{(\ell)}$ be the $n^{th}$ $\ell$-generalized Fibonacci number. Then
	\begin{align*}
	F_n^{(\ell)} = \mathrm{rnd} \left( \frac{\alpha -1}{2+ (\ell+1) (\alpha -2)} \alpha^{n-1} \right) \quad \hbox{for all} \quad n \geq 2- \ell,
	\end{align*}
	where  $\alpha$ is the unique positive dominant root and $\mathrm{rnd}(x)$ denotes the values of $x$ rounded to the nearest integer.
\end{lemma}
\section{Analytic continuation of $\ell$-generalized Fibonacci zeta function}
 \begin{theorem}\label{thm1}
 	The $\ell$-generalized Fibonacci zeta function $ \zeta_{F^{(\ell)}} (s)$ can be meromorphically continued to  whole complex plane $ \mathbb{C}$ with possible simple poles at  \\
 	\begin{equation}\label{eq10.1}
	s= s_{k, k_2, \ldots, k_{\ell}, n}=-k+\frac{2ni\pi+k_2 \log\alpha_2+\cdots+k_{\ell}\log\alpha_{\ell}}{\log\alpha}\,,\,n\in\,\mathbb{Z},k,k_2,k_3,...,k_{\ell}\in\mathbb{N}_0 ,
	\end{equation}
		$  k=k_2+k_3+\cdots+k_{\ell} ,$
 	 where $\mathbb{N}_0$ denotes the set of non-negative integers and $\log(\alpha_i) \ (i=1, 2, \ldots, \ell)$ are the principal values whose imaginary part lies in the interval $ (-\pi, \pi ].$
 	 
 	  Moreover, the residue of $  \zeta_{F^{(\ell)}}(s)$ at $ s= s_{k, k_2, \ldots, k_{\ell}, n}$ is 
 	   \begin{align*}
 	  & \left( \frac{\alpha -1}{2+ (\ell+1) (\alpha -2)} \alpha^{-1} \right)^{-s_{k, k_2, \ldots, k_{\ell}, n}}  
 	  \binom{-s_{k, k_2, \ldots, k_{\ell}, n}}{k} \left( \frac{\alpha -1}{2+ (\ell+1) (\alpha -2)} \alpha^{-1} \right)^{-k}
 	  \\& \notag
 	  \frac{k!}{k_2!\cdots\,k_\ell!}\prod_{i=2}^{\ell}\left( \frac{\alpha_i -1}{2+ (\ell+1) (\alpha_i -2)} \alpha_i^{-1} \right)^{k_i} \frac{1}{\log \alpha} \,.
 	  \end{align*}
 	 \end{theorem}
\begin{proof}
	From Lemma \ref{lem1}, we have 
	\begin{align*}
     	0 < & \left| F_n^{(\ell)} - \left( \frac{\alpha -1}{2+ (\ell+1) (\alpha -2)} \alpha^{n-1} \right) \right| < 1  \\ \notag
     	\implies  &    \left| 
	\frac{F_n^{(\ell)}}{\left( \frac{\alpha -1}{2+ (\ell+1) (\alpha -2)} \alpha^{n-1} \right)} -1 \right| < 1 \\ \notag
	 \implies  &     \left| 
	 \frac{\sum_{i=2}^{\ell} \left( \frac{\alpha_i -1}{2+ (\ell+1) (\alpha_i -2)} \alpha_i^{n-1} \right) }{\left( \frac{\alpha -1}{2+ (\ell+1) (\alpha -2)} \alpha^{n-1} \right)} \right| < 1.
	\end{align*}
	Now we consider the following.
	\begin{align*}
	\left( F_n^{(\ell)}\right)^{-s}   = &\left( \left( \frac{\alpha -1}{2+ (\ell+1) (\alpha -2)} \alpha^{n-1} \right) + \sum_{i=2}^{\ell} \left( \frac{\alpha_i -1}{2+ (\ell+1) (\alpha_i -2)} \alpha_i^{n-1} \right)  \right)^{-s}\\ \notag
	= & \left( \frac{\alpha -1}{2+ (\ell+1) (\alpha -2)} \alpha^{n-1} \right)^{-s} \left( 1+ \frac{\sum_{i=2}^{\ell} \left( \frac{\alpha_i -1}{2+ (\ell+1) (\alpha_i -2)} \alpha_i^{n-1} \right) }{\left( \frac{\alpha -1}{2+ (\ell+1) (\alpha -2)} \alpha^{n-1} \right)} \right)^{-s} \\ \notag
	= & \left( \frac{\alpha -1}{2+ (\ell+1) (\alpha -2)} \alpha^{n-1} \right)^{-s} \sum_{k=0}^{\infty} \binom{-s}{k}  \left(  \frac{\sum_{i=2}^{\ell} \left( \frac{\alpha_i -1}{2+ (\ell+1) (\alpha_i -2)} \alpha_i^{n-1} \right) }{\left( \frac{\alpha -1}{2+ (\ell+1) (\alpha -2)} \alpha^{n-1} \right)} \right)^{k}.
	\end{align*}
   	Using Proposition \ref{prop1}, observe that
	\begin{align*}
	  & \sum_{n=1}^{\infty}  \left| \left( \frac{\alpha -1}{2+ (\ell+1) (\alpha -2)} \alpha^{n-1} \right)^{-s} \sum_{k=0}^{\infty} \binom{-s}{k}  \left(  \frac{\sum_{i=2}^{\ell} \left( \frac{\alpha_i -1}{2+ (\ell+1) (\alpha_i -2)} \alpha_i^{n-1} \right) }{\left( \frac{\alpha -1}{2+ (\ell+1) (\alpha -2)} \alpha^{n-1} \right)} \right)^{k}  \right| \\ \notag
	    = &  \sum_{n=1}^{\infty}  \left|	\left( F_n^{(\ell)}\right)^{-s} \right|  < \infty.\\ \notag
		\end{align*}
		On interchanging the order of summation and using multinomial expansion formula, we get 
		\begin{align*}
	&\sum_{n=1}^{\infty} \left( F_n^{(\ell)}\right)^{-s} \\ \notag
	=& \sum_{n=1}^{\infty} \left( \frac{\alpha -1}{2+ (\ell+1) (\alpha -2)} \alpha^{n-1} \right)^{-s} \sum_{k=0}^{\infty} \binom{-s}{k}  \left(  \frac{\sum_{i=2}^{\ell} \left( \frac{\alpha_i -1}{2+ (\ell+1) (\alpha_i -2)} \alpha_i^{n-1} \right) }{\left( \frac{\alpha -1}{2+ (\ell+1) (\alpha -2)} \alpha^{n-1} \right)} \right)^{k} \\ \notag
	=& \sum_{k=0}^{\infty} \binom{-s}{k} \sum_{n=1}^{\infty} \left( \frac{\alpha -1}{2+ (\ell+1) (\alpha -2)} \alpha^{n-1} \right)^{-s}  \left(  \frac{\sum_{i=2}^{\ell} \left( \frac{\alpha_i -1}{2+ (\ell+1) (\alpha_i -2)} \alpha_i^{n-1} \right) }{\left( \frac{\alpha -1}{2+ (\ell+1) (\alpha -2)} \alpha^{n-1} \right)} \right)^{k}\\ \notag
	=&\left( \frac{\alpha -1}{2+ (\ell+1) (\alpha -2)} \alpha^{-1} \right)^{-s}\sum_{k=0}^{\infty} \binom{-s}{k}\bigg\{\left( \frac{\alpha -1}{2+ (\ell+1) (\alpha -2)} \alpha^{-1} \right)^{-k}\\&
	\sum_{k_2+k_3+\cdots+k_{\ell}=k}\frac{k!}{k_2!\cdots\,k_{\ell}!}\prod_{i=2}^{{\ell}}\left( \frac{\alpha_i -1}{2+ (\ell+1) (\alpha_i -2)} \alpha_i^{-1} \right)^{k_i}\sum_{n=1}^{\infty}\left(\frac{\alpha_2^{k_2}\cdots\,\alpha_{\ell}^{k_{\ell}}}{\alpha^{s+k}}\right)^n\bigg\}.\\ \notag
	 \end{align*}
	 Note that $ \left|  \frac{\alpha_2^{k_2} \cdots \alpha_{\ell}^{k_{\ell}} }{\alpha^{s+k}}\right| \leq   \frac{\alpha^{k_2 + \cdots + k_{\ell}}} {\alpha^{\sigma+k}} = \frac{1}{\alpha^{\sigma}} < 1$ for $\sigma >0.$
	 
	 Therefore, the above series becomes 
	 \begin{align}\label{eq14}
	 \zeta_{F^{(\ell)}}(s) =&\left( \frac{\alpha -1}{2+ (\ell+1) (\alpha -2)} \alpha^{-1} \right)^{-s}\sum_{k=0}^{\infty} \binom{-s}{k}\bigg\{\left( \frac{\alpha -1}{2+ (\ell+1) (\alpha -2)} \alpha^{-1} \right)^{-k}\\& \notag
	 \sum_{k_2+k_3+\cdots+k_\ell=k}\frac{k!}{k_2!\cdots\,k_\ell!}\prod_{i=2}^{\ell}\left( \frac{\alpha_i -1}{2+ (\ell+1) (\alpha_i -2)} \alpha_i^{-1} \right)^{k_i}\frac{1}{\alpha^{s+k}\alpha_2^{-k_2}\cdots\,\alpha_\ell^{-k_\ell}-1}\bigg\}\,.  
	 \end{align}
From \eqref{eq5}, it is clear that for any integer $i$ with  $1 \leq i\leq \ell, $ $ \left(2 + (\ell +1) (\alpha_i -2)\right) \neq 0.$ Note that the function  $  h_{k, k_2, \ldots, k_{\ell}}(s)= \frac{1}{\alpha^{s+k}\alpha_2^{-k_2}\cdots\,\alpha_\ell^{-k_\ell}-1} $ has poles at
\\
 $ s=-k+\frac{2ni\pi+k_2 \log\alpha_2+\cdots+k_\ell \log\alpha_\ell}{\log\alpha}\,,\,n\in\,\mathbb{Z},$ $ k,k_2,k_3,...,k_\ell\in\mathbb{N}_0 $ with  $ k=k_2+k_3+\cdots+k_\ell $.
 
 Consider the function $g_{k, k_2, \ldots, k_{\ell}}(s) = \alpha^{s+k}\alpha_2^{-k_2}\cdots\,\alpha_{\ell}^{-k_\ell}-1.$ Then
  $$g_{k, k_2, \ldots, k_{\ell}}'(s) = \alpha^{s+k}\alpha_2^{-k_2}\cdots\,\alpha_{\ell}^{-k_\ell} \log \alpha .$$
 Thus, $g_{k, k_2, \ldots, k_{\ell}}'(s)$ at $s=s_{k, k_2, \ldots, k_{\ell}, n}= -k+\frac{2ni\pi+k_2 \log\alpha_2+\cdots+k_\ell \log\alpha_\ell}{\log\alpha} $ is $\log \alpha$ which is non-zero. Therefore, the poles of $h_{k, k_2, \ldots, k_{\ell}}(s)$ are simple.
 
  The series \eqref{eq14} determines the holomorphic function on $\mathbb{C}$ except for the poles derived from the function  $h_{k, k_2, \ldots, k_{\ell}}(s).$ Hence, the function $\zeta_{F^{(\ell)}} (s)$ can be meromorphically continued to the whole $s$-plane and it has possible simple poles st $s = s_{k, k_2, \ldots, k_{\ell}, n} = -k+\frac{2ni\pi+k_2 \log\alpha_2+\cdots+k_\ell \log\alpha_\ell}{\log\alpha}.$
  The residue of $  h_{k, k_2, \ldots, k_{\ell}}(s)= \frac{1}{\alpha^{s+k}\alpha_2^{-k_2}\cdots\,\alpha_\ell^{-k_\ell}-1} $ at $s = s_{k, k_2, \ldots, k_{\ell}, n} = -k+\frac{2ni\pi+k_2 \log\alpha_2+\cdots+k_\ell \log\alpha_\ell}{\log\alpha}$ is 
 \begin{align*}
 \mathrm{Res}_{ s = s_{k, k_2, \ldots, k_{\ell}, n}} h_{k, k_2, \ldots, k_{\ell}}(s)=  \frac{1}{\frac{d}{ds} \left(\alpha^{s+k}\alpha_2^{-k_2}\cdots\,\alpha_\ell^{-k_\ell}-1\right) }\Bigg|_{s = s_{k, k_2, \ldots, k_{\ell}, n}} = \frac{1}{\log \alpha }.
 \end{align*}
 Now the residue of $\zeta_{F^{(\ell)}}(s)$ at $ s = s_{k, k_2, \ldots, k_{\ell}, n}$ is 
 \begin{align*}
& \mathrm{Res}_{ s = s_{k, k_2, \ldots, k_{\ell}, n}} \zeta_{F^{(\ell)}}(s) \\
=& \left( \frac{\alpha -1}{2+ (\ell+1) (\alpha -2)} \alpha^{-1} \right)^{-s_{k, k_2, \ldots, k_{\ell}, n}}  
 \binom{-s_{k, k_2, \ldots, k_{\ell}, n}}{k} \left( \frac{\alpha -1}{2+ (\ell+1) (\alpha -2)} \alpha^{-1} \right)^{-k}
\\& \notag
 \frac{k!}{k_2!\cdots\,k_\ell!}\prod_{i=2}^{\ell}\left( \frac{\alpha_i -1}{2+ (\ell+1) (\alpha_i -2)} \alpha_i^{-1} \right)^{k_i} \frac{1}{\log \alpha} \,.
 \end{align*}
\end{proof}
\begin{remark}
When $\ell =2,$ we have $\alpha \alpha_2 = -1 \implies \alpha_2 = \frac{-1}{\alpha}.$ Since $\alpha $ is a positive real number, $\arg \alpha_2 = \pi.$ From \eqref{eq10.1}, we get
 \begin{align*}
  s = s_{k, k_2, n} = &-k+\frac{2ni\pi+k\log\alpha_2}{\log\alpha} 
   =  -k+\frac{2ni\pi+k\log |\alpha_2| + i k \arg \alpha_2 }{\log\alpha} \\
    = &  -k+\frac{2ni\pi - k\log |\alpha| + i k \arg \alpha_2 }{\log\alpha} = -2k+ i\frac{2n \pi + k \arg \alpha_2 }{\log\alpha} \\
    =& -2k+ \frac{ i \pi (2n + k )}{\log\alpha}.
  \end{align*} These are the same list of poles as obtained by L. Navas \cite{Navas2001}.
\end{remark}
Next, we discuss the special cases of singularities  for our better understanding.
\begin{coro}\label{coro2}
 Let $n \in \mathbb{Z},\,k\in\mathbb{N}_0$ such that $ (\ell-1)| k $. Then the possible poles of $\zeta_{F^{(\ell)} }(s)$ are given by 
 \begin{equation} \label{20.0}
 	s_{k, \frac{k}{\ell-1},\dots,\frac{k}{\ell-1}, n}= 
 	\begin{cases}
 		-(k + \frac{k}{\ell-1}) +i\frac{\pi(2n+ \frac{k}{\ell-1})}{\log\alpha} \quad  \,\hbox{if}\,\,\ell\,\,\hbox{is even}\,,\\
 		-(k + \frac{k}{\ell-1}) +i\frac{2n\pi}{\log\alpha} \quad \quad  \quad  \hbox{if}\,\,\ell\,\,  \hbox{is odd}\,.
 	\end{cases}
 \end{equation}
\end{coro}
\begin{proof}
 Note that $ \alpha \alpha_2 \cdots \alpha_{\ell} = (-1)^{\ell + 1}.$ 
 Therefore, we have
\begin{equation}\label{20.1}
  |\alpha \alpha_2 \cdots \alpha_{\ell}| = 1 \implies  \log \alpha + \log | \alpha_2| + \cdots + \log| \alpha_{\ell}|= 0.
\end{equation}
Observe that $\alpha >1$ and for  $1 \leq j \leq \ell$,  we can write $\log \alpha_j = \log |\alpha_j| + i \arg \alpha_j$,  where $i = \sqrt{(-1)}$ and $ \arg \alpha_j  \in (-\pi, \pi].$
Since $ \log \alpha_j$'s are principal values of complex logarithms, if $\ell $ is odd, 
\begin{equation}\label{20.2}
\arg \alpha_2 + \cdots + \arg \alpha_{\ell} = 0. 
\end{equation}
If $\ell $ is even,
\begin{equation}\label{20.3}
\arg \alpha_2 + \cdots + \arg \alpha_{\ell} = \pi. 
\end{equation}
Now using \eqref{20.1}, \eqref{20.2} and \eqref{20.3} in \eqref{eq10.1} with $   k_2 =  k_3 = \cdots = k_{\ell} = \frac{k}{\ell-1} \in \mathbb{N}_0 $, we get \eqref{20.0}.

	\end{proof}
\begin{remark}
	From Corollary \ref{coro2}, we have seen that some of the possible  poles lie on the lines $ \mathrm{Re}(s) = - (k + \frac{k}{\ell-1})$ spaced at intervals of $\frac{2\pi i }{\log \alpha}.$ In particular, when $\ell $ is an odd integer, $ s =  - (k + \frac{k}{\ell-1}) , k \in \mathbb{N}_0$ are possible simple poles and when $\ell $ is an even integer, $ s =  - (k + \frac{k}{\ell-1}) + i \frac{k \pi }{(\ell -1)\log \alpha}, k \in \mathbb{N}_0$ are possible simple poles. Furthermore, when $\ell $ is an even integer and $k = -2n(\ell-1),$ then from \eqref{20.0}, we obtain $s= 2n\ell, \  n \in \mathbb{Z}_{\leq 0} $ are possible simple negative integer poles.
\end{remark}
\section{Special values at negative integer arguments}
Furthermore, we will discuss the values of $\zeta_{F^{(\ell)}}(s)$ at negative integers.
\begin{theorem}
	Let $m$ be positive integer such that $-m$ is  not a pole of $\zeta_{F^{(\ell)} } (s)$. Then $ \zeta_{F^{(\ell)} } (-m) \in \mathbb{Q}$.
\end{theorem}
\begin{proof}
	Let $-m\,,\,m\in\mathbb{N} $ is not a pole of $ \zeta_{F^{(\ell)} } (s) $.
	From \eqref{eq14}, we have 
	\begin{equation}\label{eq119}
	\begin{split}
	\zeta_{F^{(\ell)}}&(-m)=\sum_{k=0}^{m}{m\choose\,k}\left(\frac{\alpha-1}{2+(\ell+1)(\alpha-2)}\alpha^{-1}\right)^{m-k}\\&\left(\sum_{k_2+\cdots+k_{\ell}=k}\frac{k!}{k_2!\cdots\,k_{\ell}!}\prod_{i=2}^{\ell}\left(\frac{\alpha_i-1}{2+(\ell+1)(\alpha_i-2)}\alpha_i^{-1}\right)^{k_i}\frac{1}{\alpha^{-(m-k)}\alpha_2^{-k_2}\cdots\alpha_{\ell}^{-k_{\ell}}-1}\right).
	\end{split}
	\end{equation}
	Note that since $0 \leq k\leq m,$ we have $0 \leq m-k \leq m.$  Let $b $ be a fixed integer with $0 \leq b \leq m.$ If we choose $ m-k = b,$ then $ k = m-b$ and $ -m+k = -b.$ Then the coefficient of $\left(\frac{\alpha-1}{2+(\ell+1)(\alpha-2)}\alpha^{-1}\right)^{b} $ in expression \eqref{eq119} is 
	\begin{align*}
	{m\choose\,m-b}\sum_{k_2+\cdots+k_{\ell}=m-b}\frac{(m-b)!}{k_2!\cdots\,k_{\ell}!}\prod_{i=2}^{\ell}\left(\frac{\alpha_i-1}{2+(\ell+1)(\alpha_i-2)}\alpha_i^{-1}\right)^{k_i}\frac{1}{\alpha^{-b}\alpha_2^{-k_2}\cdots\alpha_{\ell}^{-k_{\ell}}-1}.\end{align*}
	
	Let $ \sigma\,:\,\mathbb{Q}(\alpha,\alpha_2,\cdots,\alpha_{\ell})\,\to \mathbb{Q}(\alpha,\alpha_2,\cdots,\alpha_{\ell}) $ be a non-trivial field automorphism.
	We know that any field automorphism permutes the roots of an irreducible polynomial. Without loss generality, let us assume that $\sigma(\alpha) = \alpha_r$ and $\sigma(\alpha_j) = \alpha$ for some $ 1 \leq j, r \leq \ell.$
	\begin{align}\label{eq2.2}
	\sigma&(\zeta_{F^{(\ell)}}(-m))=\sum_{k=0}^{m}{m\choose\,k}\left(\frac{\alpha_r-1}{2+(\ell+1)(\alpha_r-2)}\alpha_r^{-1}\right)^{m-k}\\ \notag
	&\left(\sum_{k_2+\cdots+k_{\ell}=k}\frac{k!}{k_2!\cdots\,k_{\ell}!}\prod_{i=2}^{\ell}\left(\frac{\sigma(\alpha_i)-1}{2+(\ell+1)(\sigma(\alpha_i)-2)}\sigma(\alpha_i)^{-1}\right)^{k_i}\frac{1}{\alpha_r^{-(m-k)}\sigma(\alpha_2)^{-k_2}\cdots\sigma(\alpha_{\ell})^{-k_{\ell}}-1}\right) \\ \notag
	& =\sum_{k=0}^{m}{m\choose\,k}\left(\frac{\alpha_r-1}{2+(\ell+1)(\alpha_r-2)}\alpha_r^{-1}\right)^{m-k}\\ \notag
	&\Bigg(\sum_{k_2+\cdots + k_j+ \cdots+k_{\ell}=k}\frac{k!}{k_2!\cdots\,k_{\ell}!}\prod_{i=2 , i \neq j }^{\ell}\left(\frac{\sigma(\alpha_i)-1}{2+(\ell+1)(\sigma(\alpha_i)-2)}\sigma(\alpha_i)^{-1}\right)^{k_i} \left(\frac{\alpha-1}{2+(\ell+1)(\alpha-2)}\alpha^{-1}\right)^{k_j} \\ \notag
	& \frac{1}{\alpha_r^{-(m-k)}\sigma(\alpha_2)^{-k_2}\cdots\sigma(\alpha_{\ell})^{-k_{\ell}}-1}\Bigg).
	\end{align}
	In \eqref{eq2.2}, the coefficient of $\left(\frac{\alpha-1}{2+(\ell+1)(\alpha-2)}\alpha^{-1}\right)^{b} $ is
	
	\begin{align*}
	&=\sum_{k=b}^{m}{m\choose\,k}\left(\frac{\alpha_r-1}{2+(\ell+1)(\alpha_r-2)}\alpha_r^{-1}\right)^{m-k}\\
	&\left(\sum_{k_2+\cdots+k_{j-1}+k_{j+1}\cdots+k_{\ell}=k-b}\frac{k!}{\left(\prod_{\substack{i=2\\i\neq\,j}}^{\ell}k_i!\right)b!}\prod_{\substack{i=2\\i\neq\,j}}^{\ell}\left(\frac{\alpha_i-1}{2+(\ell+1)(\alpha_i-2)}\alpha_i^{-1}\right)^{k_i}\frac{1}{\alpha^{-b}\alpha_r^{-(m-k)}\prod_{\substack{i=2\\i\neq\,j}}^{\ell}\alpha_i^{-k_i}-1}\right)\\
	&={m\choose\,m-b}\sum_{k=b}^{m}\left(\frac{\alpha_r-1}{2+(\ell+1)(\alpha_r-2)}\alpha_r^{-1}\right)^{m-k}\frac{1}{(m-k)!}\\
	&\left(\sum_{k_2+\cdots+k_{j-1}+k_{j+1}\cdots+k_{\ell}=k-b}\frac{(m-b)!}{\left(\prod_{\substack{i=2\\i\neq\,j}}^{\ell}k_i!\right)}\prod_{\substack{i=2\\i\neq\,j}}^{\ell}\left(\frac{\alpha_i-1}{2+(\ell+1)(\alpha_i-2)}\alpha_i^{-1}\right)^{k_i}\frac{1}{\alpha^{-b}\alpha_r^{-(m-k)}\prod_{\substack{i=2\\i\neq\,j}}^{\ell}\alpha_i^{-k_i}-1}\right).
	\end{align*}
	Note that $ (m-k) +(k-b) = m-b. $  Since $k$ varies from $b$ to $m,$  we have $ 0 \leq m-k \leq m-b,$ and $0 \leq k_i \leq m-b$ Thus the above sum will be 
	\begin{align*}
	{m\choose\,m-b}\sum_{k_2+\cdots+k_{\ell}=m-b}\frac{(m-b)!}{k_2!\cdots\,k_{\ell}!}\prod_{i=2}^{\ell}\left(\frac{\alpha_i-1}{2+(\ell+1)(\alpha_i-2)}\alpha_i^{-1}\right)^{k_i}\frac{1}{\alpha^{-b}\alpha_2^{-k_2}\cdots\alpha_{\ell}^{-k_{\ell}}-1}.	\end{align*}
	Note that for any integer $b$ with $0 \leq b \leq m,$ and for any field automorphism $\sigma$ we have proved that the coefficient of  $\left(\frac{\alpha-1}{2+(\ell+1)(\alpha-2)}\alpha^{-1}\right)^{b} $ in $\zeta_{F^{(\ell)}}(-m)$ is same as $\sigma(\zeta_{F^{(\ell)}}(-m))$. Therefore,
	for any field automorphism $\sigma, $ we have  $\sigma(\zeta_{F^{(\ell)}}(-m)) = \zeta_{F^{(\ell)}}(-m)$. Hence $ \zeta_{F^{(\ell)}}(-m) \in \mathbb{Q}.$
\end{proof}
\section{Concluding Remark}
Analogous to Fibonacci zeta function, it is interesting to find the zeros of $\ell$-generalized Fibonacci zeta function. The study of arithmetic nature of special values of $\ell$-generalized Fibonacci zeta function at positive even integer arguments is another important object for future discussion. 

\noindent
{\bf Data Availability Statement:}  The data used to support the findings of this study are included within the article.

\noindent
{\bf Acknowledgement:}
The authors would like to thank Dr. G. K. Viswanadham for going through this paper. First author is supported by UGC research fellowship. 
 \setlength{\parskip}{0.2 cm}

\end{document}